\documentclass[12pt]{article}
\usepackage{graphics,amsmath,amssymb}
\usepackage{amsthm}
\usepackage{amsfonts}
\usepackage{latexsym}
\usepackage{amscd}

\newtheorem{theorem}{Theorem}[section]
\newtheorem{corollary}{Corollary}[section]
\newtheorem{lemma}{Lemma}[section]

\begin{document}
\begin{center}
\title{The Asymptotic Behaviors of $\log_{r}W(r, k)$ and $\log_{k}W(r, k)$, when $W(r, k)$ is a van der Waerden Number}
\author{\textbf{Robert J. Betts}}
\maketitle
\emph{The Open University\\Postgraduate Department of Mathematics and Statistics~\footnote{During 2012--2013, when the Author was working on an earlier draft.}\\ (Main Campus) Walton Hall, Milton Keynes, MK7 6AA, UK\\
Robert\_Betts@alum.umb.edu}
\end{center}
\begin{abstract}
We derive the asymptotic behaviors of $\log_{r}W(r, k)$ and $\log_{k}W(r, k)$, when $W(r, k)$ is a van der Waerden Number. We use the approach to consider the subsets on the real line in which $W(2, 7)$ might lie.,,
\end{abstract}
\section{Introduction}
Let \(r > 1\) be any integer and $N$ any other integer much greater than $r$. Then for some positive integer exponent $n$ there exist always integers
\begin{equation}
b_{n}, b_{n - 1}, \ldots, b_{0} \in [0, r - 1],
\end{equation}
with \(1 \leq b_{n} \leq r - 1\), such that the following two statements are true always~\cite{Abramowitz},~\cite{Rosen}, for $N$:
\begin{enumerate}
\item \(N = b_{n}r^{n} + b_{n - 1}r^{n - 1}  + \cdots  +  b_{0}\).\\
\item \(r^{n} \leq N < r^{n + 1}\).
\end{enumerate}  
For example let \(N = 261, r = 10\). Then
$$
261 = 2\cdot 10^{2} + 6\cdot 10 + 1 \in [10^{2}, 10^{3}).
$$
If we let \(N = 261, r = 8\) then we get 
$$
261 = 4\cdot 8^{2} + 5 \in [8^{2}, 8^{3}).
$$
\indent Now substitute $W(r, k)$ in place of $N$, that is, now we are letting \(W(r, k) = N\) be true, where $W(r, k)$ is a van der Waerden number~\cite{van der Waerden}, and where $r$ is the number of integer colorings~\cite{Graham1},~\cite{Graham2},~\cite{Graham and Rothschild},~\cite{Graham and Spencer},~\cite{Landman and Robertson},~\cite{Landman and Culver},~\cite{Khinchin}, such that the interval $[1, W(r, k)]$ on $\mathbb{R}$ contains an arithmetic progression of $k$ terms. We then obtain through substitution by $W(r, k)$ for $N$, 
\begin{eqnarray}
N&=              &b_{n}r^{n} + b_{n - 1}r^{n - 1}  + \cdots  + b_{0} \in [r^{n}, r^{n + 1})\\
 &\Longrightarrow&W(r, k) = b_{n}r^{n} + b_{n - 1}r^{n - 1}  + \cdots  + b_{0} \in [r^{n}, r^{n + 1}).
\end{eqnarray}
For instance, for the particular case~\cite{Kouril},~\cite{Rabung and Lotts}, \(W(2, 6) = 1132\),
\begin{equation}
W(2, 6) = 1132 = 1 \cdot 2^{10} + 1 \cdot 2^{6} + 1 \cdot 2^{5} + 1 \cdot 2^{3} + 1 \cdot 2^{2},
\end{equation}
where~\cite{Betts1},~\cite{Betts2}, 
\begin{equation}
2^{10} \leq 1132 < 2^{11}.
\end{equation}
For van der Waerden numbers $W(2, 3)$, $W(2, 4)$, $W(2, 5)$, $W(2, 6)$, $\ldots$, the corresponding values for $n$ are
$$
n = 3, 5, 7, 10, \ldots,
$$
For van der Waerden numbers $W(3, 3)$, $W(3, 4)$, $\ldots$, they are
$$
n = 3, 5, \ldots,
$$
and for $W(4, 3)$, \(n = 3\).\\
\indent In two Preprints~\cite{Betts1},~\cite{Betts2}, we showed previously, among other things, the following:
\begin{enumerate}
\item \(W(r, k) \in [r^{n}, r^{n + 1})\), where, substituting $W(r, k)$ for $N$ when \(N = W(r, k)\), the exponent $n$ is that positive integer exponent for which $r^{n}$ divides $W(r, k)$, $r^{n + 1}$ does not divide $W(r, k)$ and such that $W(r, k)$ has the finite power series expansion~\cite{Abramowitz},~\cite{Rosen}, 
\begin{equation}
W(r, k) = b_{n}r^{n} + b_{n - 1}r^{n - 1} + \cdots + b_{0},
\end{equation}
where
\begin{equation}
b_{n} \in [1, r - 1], \: \: b_{n - 1}, \ldots, b_{0} \in [0, r - 1].
\end{equation}
\\
\item $W(r, k)$ also has the finite power series expansion, for some positive integer exponent $m$,
\begin{equation}
W(r, k) = c_{m}k^{m} + c_{m - 1}k^{m - 1} + \cdots + c_{0} \in [k^{m}, k^{m + 1}),
\end{equation}
where
\begin{equation}
c_{m} \in [1, k - 1], \: \: c_{m - 1}, \ldots, c_{0} \in [0, k - 1].
\end{equation}
\\
\item \(W(r, k) < r^{n + 1} \leq r^{k^{2}}\) \(\Leftrightarrow (k \geq \sqrt{n + 1}\) \(\Leftrightarrow n \leq k^{2} - 1)\)~\cite{Betts1},~\cite{Betts2}.\\
\item \(W(r, k) \approx r^{n}\) is true with a very small relative error of $|1 - O(1)|$, when $W(r, k)$, $r^{n}$, are large~\cite{Betts1}.\\
\item Since \((W(r, k) \in [r^{n}, r^{n + 1})) \wedge (W(r, k) \in [k^{m}, k^{m + 1}))\) \(\Longrightarrow W(r, k) \in [r^{n}, r^{n + 1}) \cap [k^{m}, k^{m + 1})\), we derive that
$$
W(r, k) \in [r^{n}, r^{n + 1}) \cap [k^{m}, k^{m + 1}) \not = \emptyset.
$$
\end{enumerate}
When $r$ is the base or radix we have \(n = \lfloor \log_{r}W(r, k) \rfloor\). When $k$ is the base or radix we have \(m = \lfloor \log_{k}W(r, k) \rfloor\). As an example of Statement (2),
$$
W(2, 6) = 1132 = 5\cdot 6^{3} + 1\cdot 6^{2} + 2\cdot 6^{1} + 4 \in [6^{3}, 6^{4}),
$$
where we already have expanded this van der Waerden number $W(2, 6)$ into powers of $r$ in Eqtn. (4).\\
\indent Let us discuss for a moment Statement (3), which we have treated elsewhere~\cite{Betts1},~\cite{Betts2}.\\
\indent Since we have the case \(W(3, 3) = 3^{3} \Longrightarrow 3^{3} \leq W(3, 3) < 3^{4}\), we cannot assume that \(r^{n} < W(r, k)\) is true always.  Hence we must conclude \(r^{n} \leq W(r, k) < r^{n + 1}\). Note also that \(W(r, k) = r^{n + 1}\) is impossible, since we would have the impossible result \(W(r, k) = r^{n + 1}\) \(\Longrightarrow n = \lfloor \log_{r} W(r, k) \rfloor\) \(= \lfloor n + 1 \rfloor = n + 1\). Now reconsider Eqtns. (2)--(3), then consider the two exponents
$$
n, n + 1.
$$ 
We have justified already Statement (3) elsewhere~\cite{Betts1},~\cite{Betts2}. However here we present another argument to justify Statement (3) that
$$
W(r, k) < r^{n + 1} \leq r^{k^{2}} \Leftrightarrow (k \geq \sqrt{n + 1}  \Leftrightarrow n \leq k^{2} - 1).
$$
Assume \(k^{2} \leq n < n + 1\) is true always for each and every $W(r, k)$. Then automatically
$$
k^{2} \leq n < n + 1 \Longrightarrow r^{k^{2}} \leq r^{n} \leq W(r, k) < r^{n + 1} 
$$
is true always. Yet clearly that is false, since \(r^{k^{2}} \leq r^{n} \leq W(r, k) < r^{n + 1}\) does not hold for any of the known van der Waerden numbers $W(2, 3)$, $W(2, 4)$, $W(2, 5)$, $W(2, 6)$, $W(3, 3)$, $W(3, 4)$ and $W(4, 3)$, as one can see from Table 2. Note further that if we assume \(n < k^{2} < n + 1 \Longrightarrow r^{n} < r^{k^{2}} < r^{n + 1}\) then this also is impossible, since there are neither any integers nor any integer perfect squares between the two integers $n$ and $n + 1$ where \((n, n + 1) \subset \mathbb{R}\) is open in $\mathbb{R}$. This then is another contradiction derived from our second assumption. \\
\indent On the other hand if
$$
k \geq \sqrt{n + 1} \Leftrightarrow n \leq k^{2} - 1,
$$
holds in Statement (3), a necessary and sufficient condition does exist for which \(W(r, k) < r^{n + 1} \leq r^{k^{2}}\) is true for van der Waerden number $W(r, k)$ and for positive integer $k$, namely
$$
k \geq \sqrt{n + 1} \Leftrightarrow n \leq k^{2} - 1.
$$
Now this necessary and sufficient condition \emph{does happen to hold for all the known van der Waerden numbers} $W(2, 3)$, $W(2, 4)$, $W(2, 5)$, $W(2, 6)$, $W(3, 3)$, $W(3, 4)$ and $W(4, 3)$! One should take a good look at the entries for $k$, $\sqrt{n + 1}$, $n$, $r^{n}$, $W(r, k)$, $r^{n + 1}$ and $r^{k^{2}}$, in columns two, three, four and columns seven, eight, nine and ten in Table 2, Section 2, to confirm that a necessary and sufficient condition for which \(W(r, k) < r^{n + 1} \leq r^{k^{2}}\) will be true for any given $k$ and for any given van der Waerden number $W(r, k)$ whether known or unknown, is
$$
k \geq \sqrt{n + 1} \Leftrightarrow n \leq k^{2} - 1.
$$
We provide a proof to Statement 3, this Section, in Appendix A.\\
\indent The expansion of $W(r, k)$ into powers of $r$ in (1) does not have $k$ appearing anywhere implicitly in the expansion. Also in (2) the finite expansion of $W(r, k)$ into powers of $k$ does not depend upon $r$ appearing anywhere implicitly in the expansion. Let \(N = W(r, k) \in \mathbb{N}\), \(N > k\). Then expressing $W(r, k)$ as in (1) and (2) allows us not only to express some van der Waerden numbers $W(r, N)$ recursively as discrete function values of other smaller van der Waerden numbers $W(r, k)$, such as
$$
W(r, N) = W(r, W(r, k)),
$$ 
but also to expand such a van der Waerden number like $W(r, W(r, k))$ into, for example, powers of $W(r, k)$ for some positive integer exponent $M$, such as
\begin{eqnarray}
W(r, W(r, k))&=  &C_{M}W(r, k)^{M} + C_{M - 1}W(r, k)^{M - 1} + \cdots + C_{0}\nonumber\\
             &\in&[W(r, k)^{M}, W(r, k)^{M + 1}),\nonumber
\end{eqnarray}
where $W(r, k)$ everywhere in this expansion of $W(r, W(r, k))$ into powers of $W(r, k)$ can be replaced with its expansion into powers of $k$ in (2), and where \(C_{M} \in [1, W(r, k) - 1], C_{M - 1}, \ldots, C_{0} \in [0, W(r, k) - 1]\). For example
\begin{eqnarray}
W(2, W(2, 3))&=  &C_{M}W(2, 3)^{M} + C_{M - 1}W(2, 3)^{M - 1} + \cdots + C_{0}\nonumber\\
             &\in&[W(2, 3)^{M}, W(2, 3)^{M + 1}),\nonumber
\end{eqnarray}
where everywhere $W(2, 3)$ appears within this expansion,
\begin{eqnarray}
W(2, 3)&=              &9 = 3^{2}\nonumber\\
       &\Longrightarrow&W(2, W(2, 3)) = W(2, 3^{2})\nonumber\\
       &=              &C_{M}(3^{2})^{M} + C_{M - 1}(3^{2})^{M - 1} + \cdots + C_{0}(3^{2})^{0},\nonumber
\end{eqnarray}
\(\forall \: C_{M} \in [1, 8]\), \(\forall C_{M - 1}, \ldots, C_{0} \in [0, 8]\), \(r = 2, k = 3\). The same thing can be done if for integer \(R > r\), \(R = W(r, k) \Longrightarrow W(R, k) = W(W(r, k), k)\). That is, one can expand the van der Waerden number $W(W(r, k), k)$ into powers of $W(r, k)$ just as one can expand the van der Waerden number $W(r, W(r, k))$ into powers of $W(r, k)$.\\
\indent By the compound proposition in (3) we see that a necessary and sufficient condition for which $W(r, k)$ is bounded above by $r^{k^{2}}$, is for \(k \geq \sqrt{n + 1} \Leftrightarrow n \leq k^{2} - 1\) to hold~\cite{Betts1},\cite{Betts2}. The last result in (4) leads us to investigate the behavior of $\log_{r}W(r, k)$ (See Table 1 and Table 2). So in the next Section we demonstrate the asymptotic behaviors of both $\log_{r}W(r, k)$ and $\log_{k}W(r, k)$, when $r$ and $k$, respectively, increase each in turn without bound. However before we leave this Section we provide a Lemma to justify more formally Eqtn. (3) and which we can use in the Proof to Theorem 3.1. 
\begin{lemma}
Let \(N = b_{n}r^{n} + b_{n - 1}r^{n - 1} + \cdots + b_{0} > r\) and \(W(r, k) = N\), where
$$
N = b_{n}r^{n} + b_{n - 1}r^{n - 1} + \cdots + b_{0},
$$
is the finite expansion of the integer $N$ into powers of $r$ and where the integers \(b_{n}, b_{n - 1}, \ldots, b_{0}\) are as in Eqtn. (7), so that $(b_{n}b_{n - 1}\cdots b_{0})_{r}$ would be the base $r$ representation of $N$ (See~\cite{Rosen},~\cite{Abramowitz}). Then for any van der Waerden number $W(r, k)$ such that, after an $r$--coloring among the integers in $[1, W(r, k)]$, the interval $[1, W(r, k)]$ has an AP of $k$ terms, 
$$
W(r, k) \in [r^{n}, r^{n + 1}) \subset \mathbb{R}.
$$
\end{lemma}
\begin{proof}
\begin{eqnarray}
W(r, k)&=              &N \: \: \mbox{(From what is given)},\nonumber\\
N      &=              &b_{n}r^{n} + b_{n - 1}r^{n - 1} + \cdots + b_{0} \: \: \mbox{(From what is given)},\nonumber\\
       &               &b_{n}r^{n} + b_{n - 1}r^{n - 1} + \cdots + b_{0} \in [r^{n}, r^{n + 1}) \subset \mathbb{R} \: \mbox{(See~\cite{Rosen},~\cite{Abramowitz})},\nonumber\\
       &\Longrightarrow&N \in [r^{n}, r^{n + 1}) \subset \mathbb{R} \: \: \mbox{(By substitution with $N$ for the expansion)}\nonumber\\
       &\therefore     &W(r, k) = b_{n}r^{n} + b_{n - 1}r^{n - 1} + \cdots + b_{0} \in [r^{n}, r^{n + 1}) \subset \mathbb{R}.\nonumber\\
       &               &\: \: \mbox{(By substitution with $W(r, k)$ for $N$)}\nonumber
\end{eqnarray}
\end{proof}
\section{Asymptotic Behaviors of $\log_{r}W(r, k)$, $\log_{k}W(r, k)$}
We establish the asymptotic behavior of $\log_{r}W(r, k)$ with Theorem 2.1 and the asymptotic behavior of $\log_{k}W(r, k)$ with Theorem 2.2.
\begin{theorem}
Let 
\begin{equation}
W(r, k) = b_{n}r^{n} + b_{n - 1}r^{n - 1} + \cdots + b_{0}.
\end{equation}
Then for large $r$, $r^{n}$ and $W(r, k)$, \(\log_{r}W(r, k) = n + O(1)\).
\end{theorem}
\begin{proof}
By substitution with the right hand side of Equation (10) for $W(r, k)$,
\begin{eqnarray}
\lim_{r \rightarrow \infty}\log_{r}W(r, k)&=&\lim_{r \rightarrow \infty}\log_{r}(b_{n}r^{n} + b_{n - 1}r^{n - 1} + \cdots + b_{0})\\
                                          &=&\lim_{r \rightarrow \infty}\log_{r}b_{n}r^{n}\left(1 + \frac{b_{n - 1}}{b_{n}r} + \frac{b_{n - 2}}{b_{n}r^{2}} + \cdots + \frac{b_{0}}{b_{n}r^{n}}\right)\nonumber\\
                                          &=&\lim_{r \rightarrow \infty}\log_{r}b_{n}\\
                                          &+&\lim_{r \rightarrow \infty} n\log_{r}r + \lim_{r \rightarrow \infty} \log_{r}\left(1 + \frac{b_{n - 1}}{b_{n}r} + \frac{b_{n - 2}}{b_{n}r^{2}} + \cdots + \frac{b_{0}}{b_{n}r^{n}}\right).\nonumber\\
                                          & &                                                                                                                                                                                                                            
\end{eqnarray}
Look at the first limit in Eqtns. (12)--(13). This is $O(1)$, since \(b_{n} \leq r - 1 \Longrightarrow \lim_{r \rightarrow \infty}\log_{r}b_{n}\) \(\leq \lim_{r \rightarrow \infty}\log_{r}r = 1\). Now look at the second limit in Eqtns. (12)--(13). This limit simply is $n$, since \(\lim_{r \rightarrow \infty}\log_{r}r = 1 \Longrightarrow\) \(\lim_{r \rightarrow \infty}n\log_{r}r = n\lim_{r \rightarrow \infty}\log_{r}r = n\). Finally we find that 
\begin{equation}
\lim_{r \rightarrow \infty} \log_{r}\left(1 + \frac{b_{n - 1}}{b_{n}r} + \frac{b_{n - 2}}{b_{n}r^{2}} + \cdots + \frac{b_{0}}{b_{n}r^{n}}\right),
\end{equation}
is equal to or smaller than the limit
\begin{equation}
\lim_{r \rightarrow \infty} \log_{r}1 + \lim_{r \rightarrow \infty}\left(\frac{b_{n - 1}}{b_{n}r} + \frac{b_{n - 2}}{b_{n}r^{2}} + \cdots + \frac{b_{0}}{b_{n}r^{n}}\right) = 0 + O(1),
\end{equation}
since obviously as $r$, $r^{n}$ increase without bound,
\begin{equation}
\left(\frac{b_{n - 1}}{b_{n}r} + \frac{b_{n - 2}}{b_{n}r^{2}} + \cdots + \frac{b_{0}}{b_{n}r^{n}}\right) \leq 1.
\end{equation}
Define
\begin{eqnarray}
\varepsilon_{1}(r, n)&=& \log_{r}b_{n}, \: \: 1 \leq b_{n} \leq r - 1,\nonumber\\
\varepsilon_{2}(r, n)&=& \log_{r}\left(1 + \frac{b_{n - 1}}{b_{n}r} + \frac{b_{n - 2}}{b_{n}r^{2}} + \cdots + \frac{b_{0}}{b_{n}r^{n}}\right),\nonumber\\
                  & &b_{n - 1}, b_{n - 2}, \ldots, b_{0} \in [0, r - 1].\nonumber
\end{eqnarray}
So using these three results in Eqtns. (11)--(13) we derive as $r$, $r^{n}$, $W(r, k)$, increase without bound,
\begin{eqnarray}
\log_{r}W(r, k)&=&\varepsilon_{1}(r, n) + n + \varepsilon_{2}(r, n)\\ 
               &=&n + O(1),
\end{eqnarray}
where $\varepsilon_{1}(r, n)$, $\varepsilon_{2}(r, n)$, are two, positive discrete function values such that \(\varepsilon_{1}(r, n) = O(1)\) and \(\varepsilon_{2}(r, n) = O(1)\). We get the final result using the fact~\cite{Erdelyi}, that \(O(1) + O(1) = O(1)\).
\end{proof}
Since \(\log_{r} W(r, k) = \delta(r, k) \Longrightarrow W(r, k) = r^{\delta(r, k)}\), the significance of Theorem 2.1 is that as $r^{n}$, $W(r, k)$ grow larger, \(|\delta(r, k) - n| = O(1)\). One should not be too astonished by this. Since on the interval $[n, n + 1)$ on $\mathbb{R}$ we have both \(|n + 1 - n| = 1\) and \(\delta(r, k) \in [n, n + 1)\) for each van der Waerden number \(W(r, k) = r^{\delta(r, k)}\), we have also in fact that \(|\delta(r, k) - n| \leq 1\) for each van der Waerden number $W(r, k)$ (Compare the values for $\delta(r, k)$ and $n$ in Table 1 and Table 2, this paper). \\
\indent The result to find the asymptotic behavior of $\log_{k}W(r, k)$ is straightforward and very similar to the approach for the proof to Theorem 2.1.
\begin{theorem}
Let 
\begin{equation}
W(r, k) = c_{m}k^{m} + c_{m - 1}k^{m - 1} + \cdots + c_{0} \in [k^{m}, k^{m + 1}) \in \mathbb{R},
\end{equation}
be the expansion of $W(r, k)$ into powers of $k$, where
\begin{equation}
c_{m} \in [1, k - 1], \: \: c_{m}, c_{m - 1}, \ldots, c_{0} \in [0, k - 1].
\end{equation}
Then for large $k$, $k^{m}$ and $W(r, k)$ and as these increase without bound, \(\log_{k}W(r, k) = m + O(1)\).
\end{theorem}
\begin{proof}
By substitution with the right hand side of
\begin{equation}
W(r, k) = c_{m}k^{m} + c_{m - 1}k^{m - 1} + \cdots + c_{0}
\end{equation}
for $W(r, k)$,
\begin{eqnarray}
\lim_{k \rightarrow \infty}\log_{k}W(r, k)&=&\lim_{k \rightarrow \infty}\log_{k}(c_{m}k^{m} + c_{m - 1}k^{m - 1} + \cdots + c_{0})\\
                                          &=&\lim_{k \rightarrow \infty}\log_{k}c_{m}k^{m}\left(1 + \frac{c_{m - 1}}{c_{m}k} + \frac{c_{m - 2}}{c_{m}k^{2}} + \cdots + \frac{c_{0}}{c_{m}k^{m}}\right)\nonumber\\
                                          &=&\lim_{k \rightarrow \infty}\log_{k}c_{m}\\
                                          &+&\lim_{k \rightarrow \infty} m\log_{k}k\\
                                          &+&\lim_{k \rightarrow \infty} \log_{k}\left(1 + \frac{c_{m - 1}}{c_{m}k} + \frac{c_{m - 2}}{c_{m}k^{2}} + \cdots + \frac{c_{0}}{c_{m}k^{m}}\right).
\end{eqnarray}
Again as in the proof to Theorem 2.1, the first limit in Eqtns. (23) is $O(1)$, since \(c_{m} \leq k - 1 \Longrightarrow \lim_{k \rightarrow \infty}\log_{k}c_{m}\) \(\leq \lim_{k \rightarrow \infty}\log_{k}k = 1\). Now look at the second limit in Eqtns. (24). Again as in the proof to the previous Theorem, this limit simply is $m$, since \(\lim_{k \rightarrow \infty}m\log_{k}k = m\lim_{k \rightarrow \infty}\log_{k}k = m\). Third in Eqtn. (25),
\begin{equation}
\lim_{k \rightarrow \infty} \log_{k}\left(1 + \frac{c_{m - 1}}{c_{m}k} + \frac{c_{m - 2}}{c_{m}k^{2}} + \cdots + \frac{c_{0}}{c_{m}k^{m}}\right),
\end{equation}
is equal to or smaller than the limit
\begin{equation}
\lim_{k \rightarrow \infty} \log_{k}1 + \lim_{k \rightarrow \infty}\left(\frac{c_{m - 1}}{c_{m}k} + \frac{c_{m - 2}}{c_{m}k^{2}} + \cdots + \frac{c_{0}}{c_{m}k^{m}}\right) = 0 + O(1),
\end{equation}
since as $k$, $k^{m}$ increase without bound,
\begin{equation}
\left(\frac{c_{m - 1}}{c_{m}k} + \frac{c_{m - 2}}{c_{m}k^{2}} + \cdots + \frac{c_{0}}{c_{m}k^{m}}\right) \leq 1.
\end{equation}
Define
\begin{eqnarray}
\eta_{1}(k, m)&=&\log_{k}c_{m}, \: \: 1 \leq c_{m} \leq k - 1,\nonumber\\
\eta_{2}(k, m)&=&\log_{k}\left(1 + \frac{c_{m - 1}}{c_{m}k} + \frac{c_{m - 2}}{c_{m}k^{2}} + \cdots + \frac{c_{0}}{c_{m}k^{m}}\right),\nonumber\\
              & &c_{m - 1}, c_{m - 2}, \ldots, c_{0} \in [0, k - 1].\nonumber
\end{eqnarray}
Therefore in Eqtns. (21)--(25) and similar to how we derived the result for Theorem 2.1,
\begin{eqnarray}
\log_{k}W(r, k)&=&\eta_{1}(k, m) + m + \eta_{2}(k, m)\\
               &=&m + O(1),
\end{eqnarray}
for two positive discrete function values $\eta_{1}(k, m)$, $\eta_{2}(k, m)$, such that \(\eta_{1}(k, m) = O(1), \eta_{2}(k, m) = O(1)\).
\end{proof}
Again as with Theorem 2.1, since \(\log_{k} W(r, k) = \delta(r, k) \Longrightarrow W(r, k) = k^{\delta(r, k)} \geq k^{m}\), the significance of Theorem 2.2 is that as $k^{m}$, $W(r, k)$ grow larger, \(|\delta(r, k) - m| = O(1)\).
\begin{center}
\begin{tabular}{|l      |c      |c          |c                                       |c                         |c                                       |c                                        |r|}
\hline
               $r$   &  $k$  &  $n$   &     \(\delta(r, k) = \log_{r} W(r, k)\)   &   \(W(r, k) = N\)     &      \(N = r^{\delta(r, k)}\)        &         \(\delta(r, k) \in [n, n + 1)\)  \\   
\hline
                $2$  &  $3$  &  $3$   &      $3.17010\ldots$                      &   \(W(2, 3) = 9\)     &       \(9 = 2^{3.17010\ldots}\)      &         \(3.17010\ldots \in [3, 4)\)          \\
 
                $2$  &  $4$  &  $5$   &      $5.12963\ldots$                      &   \(W(2, 4) = 35\)    &       \(35 = 2^{5.12963\ldots}\)     &         \(5.12963\ldots \in [5, 6)\)          \\
                
                $2$  &  $5$  &  $7$   &      $7.47623\ldots$                      &   \(W(2, 5) = 178\)   &       \(178 = 2^{7.47623\ldots}\)    &         \(7.47623\ldots \in [7, 8)\)          \\ 

                $2$  &  $6$  &  $10$  &      $10.14534\ldots$                     &   \(W(2, 6) = 1132\)  &       \(1132 = 2^{10.14534\ldots}\)  &         \(10.14534\ldots \in [10, 11)\)      \\
                 
                $3$  &  $3$  &  $3$   &      $3.00002\ldots$                      &   \(W(3, 3) = 27\)    &       \(27 = 3^{3.00002\ldots}\)     &         \(3.00002\ldots \in [3, 4)\)        \\
                
                $3$  &  $4$  &  $5$   &      $5.17037\ldots$                      &   \(W(3, 4) = 293\)   &       \(293 = 3^{5.17037\ldots}\)    &         \(5.17037\ldots \in [5, 6)\)     \\

                $4$  &  $3$  &  $3$   &      $3.12417\ldots$                      &   \(W(4, 3) = 76\)    &       \(76 = 4^{3.12417\ldots}\)     &         \(3.12417\ldots \in [3, 4)\)     \\
  
\hline
\end{tabular}
\end{center}
\begin{center}
Table 1. \textbf{Note}: All logarithms here are taken to the respective base $r$.
\end{center}
\begin{center}
\begin{tabular}{|l      |c      |c                  |c      |c            |c              |c                  |c              |c              |r|}
\hline
                $r$  &  $k$  &  $\sqrt{n + 1}$  &   $n$  &  $\log r$  &   $\log k$    &   $r^{n}$     &       $W(r, k)$   &   $r^{n + 1}$  &  $r^{k^{2}}$ \\   
\hline
                $2$  &  $3$  &  $2$             &   $3$  &  $0.6931$  &   $1.0986$    &   $2^{3}$     &       $9$         &   $2^{4}$      &  $2^{9}$  \\
 
                $2$  &  $4$  &  $2.449\ldots$   &   $5$  &  $0.6931$  &   $1.3862$    &   $2^{5}$     &       $35$        &   $2^{6}$      &  $2^{16}$  \\
                
                $2$  &  $5$  &  $2.828\ldots$   &   $7$  &  $0.6931$  &   $1.6094$    &   $2^{7}$     &       $178$       &   $2^{8}$      &  $2^{25}$ \\ 

                $2$  &  $6$  &  $3.316\ldots$   &   $10$ &  $0.6931$  &   $1.7917$    &   $2^{10}$    &       $1132$      &   $2^{11}$     &  $2^{36}$  \\
                 
                $3$  &  $3$  &  $2$             &   $3$  &  $1.0986$  &   $1.0986$    &   $3^{3}$     &       $27$        &   $3^{4}$      &  $3^{9}$  \\
                
                $3$  &  $4$  &  $2.449\ldots$   &   $5$  &  $1.0986$  &   $1.3862$    &   $3^{5}$     &       $293$       &   $3^{6}$      &  $3^{16}$  \\

                $4$  &  $3$  &  $2$             &   $3$  &  $1.3862$  &   $1.0986$    &   $4^{3}$     &       $76$        &   $4^{4}$      &  $4^{9}$  \\
  
\hline
\end{tabular}
\end{center}
\begin{center}
Table 2. Logarithms taken to the base $e$.
\end{center}
\section{A lower Bound for the integer Exponent $n$}
Let $a$, \(d > 1\) be any two positive integers such that
\begin{equation}
a + id \in [1, W(r, k)], \: \: \forall \: i \in [0, k - 1]. 
\end{equation}
That is, the interval $[1, W(r, k)]$ on $\mathbb{R}$ contains the arithmetic progression \(a, a + d, a + 2d, \dots, a + (k - 1)d\), where \(a + (k - 1)d \leq W(r, k)\). Since \(W(r, k) \in [r^{n}, r^{n + 1})\), we then can derive a lower bound on the positive integer exponent $n$ for any integer $a$ in $[1, W(r, k)]$ for which 
\begin{equation}
a + (k - 1)d \leq W(r, k), 
\end{equation}
as we do for the following Theorem (Here we can take the logarithm to any base, including to the base $e$, or to base $10$, etc.).
\begin{theorem}
Let
$$
a, a + d, a + 2d, \cdots, a + (k - 1)d,
$$
be any arithmetic progression with $k$ terms, composed from any integer elements $a + id$, \(0 \leq i \leq k - 1\), all taken from the interval $[1, W(r, k)]$. Then 
\begin{equation}
n > \frac{\log\left(a + \frac{(k - 1)d}{2}\right)}{\log r} - 1.
\end{equation}
Moreover, \(d \leq \frac{W(r, k) - a}{k - 1}\).
\end{theorem}
\begin{proof}
\begin{eqnarray}
a + id \in [1, W(r, k)]&\Longrightarrow&a + id \leq W(r, k) \: \forall \: i \in [0, k - 1]\\
                       &\Longrightarrow&\underbrace{a + a + d + a + 2d + \cdots + a + (k - 1)d}_{\mbox{k terms}}\nonumber\\
                       &\leq           &\underbrace{W(r, k) + W(r, k) + \cdots + W(r, k)}_{\mbox{k terms}} < kr^{n + 1}\nonumber\\
                       &\Longrightarrow&ka + \underbrace{(1 + 2 + 3 + \cdots + k - 1)d}_{\mbox{k terms}}\\
                       &\leq           &kW(r, k) < kr^{n + 1}\nonumber\\                                                 
                       &\Longrightarrow&ka + \frac{(k - 1)kd}{2} \leq kW(r, k) < kr^{n + 1},
\end{eqnarray}
where in Eqtn. (36) \(W(r, k) < r^{n + 1}\) follows from Lemma 1.1 in Section 1 (See also Eqtns. (2)--(3) and Eqtns. (6)--(7)), and the result
\begin{equation}
1 + 2 + 3 + \cdots + k - 1 = \frac{(k - 1)k}{2},
\end{equation}
we use to evaluate the finite sum in Eqtn. (35). So dividing by $k$ in Eqtn. (36) we derive
\begin{eqnarray}
a + \frac{(k - 1)d}{2}&\leq           &W(r, k) < r^{n + 1}\\
                      &\Longrightarrow&\log\left(a + \frac{(k - 1)d}{2}\right) \leq \log W(r, k) < (n + 1)\log r\nonumber\\
                      &\Longrightarrow&\log\left(a + \frac{(k - 1)d}{2}\right) < (n + 1)\log r\\
                      &\Longrightarrow&\frac{\log\left(a + \frac{(k - 1)d}{2}\right)}{\log r} < n + 1 \nonumber\\
                      &\Longrightarrow&\frac{\log\left(a + \frac{(k - 1)d}{2}\right)}{\log r} - 1 < n.
\end{eqnarray}
Finally since \(a + (k - 1)d \leq W(r, k)\) must hold since \(a + (k - 1)d \in [1, W(r, k)]\) as given, we obtain \(a + (k - 1)d \leq W(r, k)\) \(\Longrightarrow d \leq \frac{W(r, k) - a}{k - 1}\).   
\end{proof}
\subsection{Application of these Results to $W(2, 7)$}
At present the actual value of $W(2, 7)$ is unknown. Yet J. Rabung and M. Lotts~\cite{Rabung and Lotts} have indicated that $W(2, 7)$ has a lower bound of $3703$.\\
\indent Theorem 3.1 allows us to characterize a lower bound on $n$, where by Lemma 1.1,
\begin{equation}
W(2, 7) \in [2^{n}, 2^{n + 1}).
\end{equation} 
Let \(\log_{2}W(2, 7) = \delta(2, 7) \Longrightarrow W(2, 7) = 2^{\delta(2, 7)}\). We know automatically then that \(\delta(2, 7) \in [n, n + 1)\) for some positive integer exponent $n$ such that Eqtn. (41) holds. Since \(\log_{2}3703 = 11.8544\cdots\) we also know that \(3703 = 2^{11.8544\cdots} < 2^{\delta(2, 7)} < 2^{n + 1}\) \(\Longrightarrow \delta(2, 7) > 11.8544\cdots\). In a previous result~\cite{Betts1}, we showed that for the unknown van der Waerden number $W(2, 7)$ the positive exponent values of $n$, $\delta(2, 7)$ for which \(2^{n} \leq W(2, 7) < 2^{n + 1}\) and \(W(2, 7) = 2^{\delta(2, 7)}\) are true are such that \(n, \delta(2, 7) \in [11, 48]\). To obtain the lower bound of eleven here on $n$, $\delta(2, 7)$ (i.e., for \(W(2, 7) = 2^{\delta(2, 7)} \geq 2^{n}\) to be true) we derive, using the lower bound $3703$ on $W(2, 7)$ by J. Rabung and M. Lotts~\cite{Rabung and Lotts},~\cite{Betts1},
$$
n > \frac{\log_{2} W(2, 7)}{\log_{2} 2} - 1 > \frac{\log_{2} 3703}{\log_{2} 2} - 1 = 10.8544\cdots \Longrightarrow n \geq 11, 
$$
and since \(2^{\delta(2, 7)} > 3703\) \(\Longrightarrow \delta(2, 7) > 11.8544\cdots\) \(\Longrightarrow \delta(2, 7) > 11.8544\cdots \geq 11\). \\
\indent J. Rabung and M. Lotts~\cite{Rabung and Lotts}, have indicated that \(W(2, 8) > 11495\). Since \(11495 = 2^{13.4896\cdots} < 2^{14}\), here we show how to find the subset on $\mathbb{R}$ in which $W(2, 7)$ lies, if we have a \emph{valid assumption} about an upper bound on $n$, namely if we can assume that \(n < 16\). 
\begin{theorem}
Let \(W(2, 7) > 3703\), \(W(2, 7) = 2^{\delta(2, 7)}\) and \(n \geq 11, \delta(2, 7) > \log_{2}3703\). Then the following two statements are equivalent when applied to $W(2, 7)$:
\begin{enumerate}
\item 
\begin{equation}
n, \delta(2, 7) \in [11, 16) = [11, 12) \cup [12, 13) \cup [13, 14] \cup [14, 15) \cup [15, 16).
\end{equation}
\\
\item
\begin{equation}
2^{n}, W(2, 7) \in [2^{11}, 2^{16}) = [2^{11}, 2^{12})\cup [2^{12}, 2^{13})\cup [2^{13}, 2^{14}) \cup [2^{14}, 2^{15}) \cup [2^{15}, 2^{16}).
\end{equation}
\end{enumerate}
\end{theorem}
\begin{proof}
\noindent ((2) $\Longrightarrow$ (1)): Assume Eqtn. (43) holds, where \(W(2, 7) = 2^{\delta(2, 7)}\). Then 
\begin{eqnarray}
2^{n}, 2^{\delta(2, 7)}&\in            &[2^{11}, 2^{12})\cup [2^{12}, 2^{13})\cup [2^{13}, 2^{14}) \cup [14, 15) \cup [15, 16)\nonumber\\
                       &               &= [2^{11}, 2^{16}) \nonumber\\
                       &\Longrightarrow&2^{11} \leq 2^{n}, 2^{\delta(2, 7)} < 2^{16}\nonumber\\
                       &\Leftrightarrow&11 \leq n, \delta(2, 7) < 16\nonumber\\
                       &\Leftrightarrow&n, \delta(2, 7) \in [11, 16) = [11, 12)\cup [12, 13) \cup [13, 14) \cup [14, 15) \cup [15, 16).\nonumber\\
\end{eqnarray}
\noindent ((1) $\Longrightarrow$ (2)): Assume Eqtn. (42) holds. Then 
\begin{eqnarray}
n, \delta(2, 7)&\in           &[11, 16) = [11, 12) \cup [12, 13) \cup [13, 14] \cup [14, 15) \cup [15, 16)\nonumber\\
               &\Longrightarrow&11 \leq n < 16, 11 < \log_{2}3703 < \delta(2, 7) < 16\nonumber\\
               &\Leftrightarrow&2^{11} \leq 2^{n}, 2^{\delta(2, 7)} < 2^{16}\nonumber\\
               &\Leftrightarrow&2^{n}, W(2, 7) \in [2^{11}, 2^{16})\nonumber\\
               &=              &[2^{11}, 2^{12})\cup [2^{12}, 2^{13})\cup [2^{13}, 2^{14}) \cup [2^{14}, 2^{15}) \cup [2^{15}, 2^{16}),\nonumber
\end{eqnarray} 
since \(\log_{2}W(2, 7) = \delta(2, 7) \Longrightarrow W(2, 7) = 2^{\delta(2, 7)}\).
\end{proof}
For the upper bound of forty--eight on $n$ when \(r = 2, k = 7\), we refer the Reader to our prior results~\cite{Betts1},~\cite{Betts2}. These results and previous work~\cite{Betts1} (See Table A) lead us to ask is \(\delta(2, 7) = \log_{2} W(2, 7) \in [11, 15)\) possible? Since 
\begin{equation}
3703 < W(2, 7) < r^{n + 1},
\end{equation}
and since, using the result \(3703 < W(2, 7)\) indicated by J. Rabung and M. Lotts~\cite{Rabung and Lotts} we have also
\begin{equation}
3703 = 2^{11} + 2^{10} + 2^{9} + 2^{6} + 2^{5} + 2^{4} + 2^{2} + 2^{1} + 2^{0} < W(2, 7) < 2^{n + 1},
\end{equation}
it follows that, since \(\log_{2}3703 = 11.854\cdots > 11\), either
\begin{equation}
2^{11} \leq 2^{n} < 2^{11} + 2^{10} + 2^{9} + 2^{6} + 2^{5} + 2^{4} + 2^{2} + 2^{1} + 2^{0} < W(2, 7) < 2^{n + 1},
\end{equation}
or else
\begin{equation}
2^{11} \leq 2^{11} + 2^{10} + 2^{9} + 2^{6} + 2^{5} + 2^{4} + 2^{2} + 2^{1} + 2^{0} < 2^{n} \leq W(2, 7) < 2^{n + 1},
\end{equation}
is true. In fact if we assume that even \(\delta(2, 7) = \log_{2} W(2, 7) \in [11, 16)\) is a valid assumption, it would follow from our approach (This paper. See also~\cite{Betts1},~\cite{Betts2}), that either \(W(2, 7) \in (3703, 2^{12})\), \(W(2, 7) \in [2^{12}, 2^{13})\), \(W(2, 7) \in [2^{13}, 2^{14})\), \(W(2, 7) \in [2^{14}, 2^{15})\) or else \(W(2, 7) \in [2^{15}, 2^{16})\). Actually we have that
$$
7 \geq \sqrt{n + 1} \Leftrightarrow 11 \leq n \leq 48 \Longrightarrow 3703 < 2^{n} \leq W(2, 7) < 2^{n + 1} \leq 2^{49}, 
$$
where
$$
W(2, 7) = 2^{n} + b_{n - 1}2^{n - 1} + b_{n - 2}2^{n - 2} + \ldots + b_{0},
$$
is true for some combination of nonnegative integers \(b_{n} = 1\),\(b_{n - 1}, b_{n - 2}, \ldots, b_{0} \in \{0, 1\}\).\\
\indent The possible intervals in which $W(2, 7)$ lies are listed in Appendix B.\\
\indent With regard to computational complexity it ought to be easier given the right algorithms, either to find or to estimate the positive real exponent $\delta(r, k)$, then to find or to estimate the value of each integer $W(r, k)$~\cite{Rabung and Lotts}. \\ 
\indent Next we prove a Corollary to establish a lower bound on $n$ given any integers \(a, a + d, \ldots, a + 6d \in [1, W(2, 7)]\), where \(W(2, 7) \in [2^{n}, 2^{n + 1})\).
\begin{corollary}
Suppose that for positive integer $a$ and common difference \(d > 1\), the AP
\begin{equation}
a, a + d, a + 2d, a + 3d, a + 4d, a + 5d, a + 6d,
\end{equation}
with each integer an element in $[1, W(2, 7)]$, is an AP in the interval $[1, W(2, 7)]$ where \(a + 6d \leq W(2, 7)\). Then \(n > \log_{2}(a + 3d) - 1\), where \(d \leq \frac{W(2, 7) - a}{6}\).
\end{corollary}
\begin{proof}
This follows at once from Theorem 3.1, when in Eqtn. (40), \(k = 7, r = 2\) and the logarithms are taken base two, since then the denominator in Eqtn. (40) becomes \(log_{2}2 = 1\).
\end{proof}
\section{A Lower Bound on the Exponent $n$ expressed in Terms of $m$, $k$, and $r$}
Let
\begin{eqnarray}
W(r, k)&=&b_{n}r^{n} + b_{n - 1}r^{n - 1} + \cdots + b_{0} \in [r^{n}, r^{n + 1}),\\
W(r, k)&=&c_{m}k^{m} + c_{m - 1}k^{m - 1} + \cdots + c_{0} \in [k^{m}, k^{m + 1}),
\end{eqnarray}
be the two respective expansions of $W(r, k)$ into powers of $r$ and $k$. Here we find a lower bound on $n$ expressed in terms of $m$, $k$ and $r$. Logarithms are taken to any arbitrary base. 
\begin{theorem}
\begin{equation}
n > \frac{m\log k}{\log r} - 1.
\end{equation}
\end{theorem}
\begin{proof}
From Eqtns. (50)--(51), we have both that 
\begin{equation}
k^{m} \leq W(r, k),
\end{equation}
and
\begin{equation}
W(r, k) < r^{n + 1},
\end{equation}
from which we derive
\begin{eqnarray}
k^{m}&\leq           &W(r, k) < r^{n + 1}\\
     &\Longrightarrow&m\log k \leq \log W(r, k) < (n + 1)\log r\nonumber\\
     &\Longrightarrow&\frac{m\log k}{\log r} \leq \frac{\log W(r, k)}{\log r} < n + 1\nonumber\\
     &\Longrightarrow&\frac{m\log k}{\log r} - 1 \leq \frac{\log W(r, k)}{\log r} - 1 < n\nonumber\\
     &\Longrightarrow&\frac{m\log k}{\log r} - 1 < n.
\end{eqnarray}
\end{proof}
In a previous paper~\cite{Betts2}, we treated the result
\begin{equation}
n > \frac{\log W(r, k)}{\log r} - 1.
\end{equation}
However since \(k^{m} \leq W(r, k) < k^{m + 1}\) one also can convince oneself that
\begin{equation}
m > \frac{\log W(r, k)}{\log k} - 1.
\end{equation}
\section{The Limits \(n \rightarrow \infty \sqrt[n]{W(r, k)}\) and \(m \rightarrow \infty \sqrt[m]{W(r,k)}\)}
We demonstrate these with the following Theorem. Note that ordinarily, \(\sqrt[n]{W(r, k)} \not = \sqrt[m]{W(r, k)}\) is true infinitely often and all the respective, corresponding terms in $n$, $m$,
$$
\sqrt[n]{W(r, k)}, \sqrt[m]{W(r, k)},
$$ 
are not identical except when \(n = m, r = k\). In fact let \(n \in \{n_{1}, n_{2}, \ldots\} \subset \mathbb{N}\) and \(m \in \{m_{1}, m_{2}, \ldots\} \subset \mathbb{N}\), where \(n_{1} < n_{2} < \cdots\), \(m_{1} < m_{2} < \cdots\). Then the two sequences of integer exponents \((n)_{n = n_{1}}^{\infty}, (m)_{m = m_{1}}^{\infty}\) are not identical, certainly not for example when \(r \not = k, r \ll k, r \gg k\) hold. Even the two intervals $[r^{n}, r^{n + 1})$, $[k^{m}, k^{m + 1})$, on $\mathbb{R}$ are not identical when \(r \not = k\). In fact \((n - m)_{n = n_{1}, m = m_{1}}^{\infty} \not = 0\) is true infinitely often. One should be careful to keep this in mind for the next Theorem.
\begin{theorem}
Let, respectively,
\begin{equation}
W(r, k) = b_{n}r^{n} + b_{n - 1}r^{n - 1} + \cdots + b_{0} \in [r^{n}, r^{n + 1}),
\end{equation}
and
\begin{equation}
W(r, k) = c_{m}k^{m} + c_{m - 1}k^{m - 1} + \cdots + c_{0} \in [k^{m}, k^{m + 1}),
\end{equation}
where either at least \(n \gg r, r \not = k\) is true for the expansion into powers of $r$ or else \(m \gg k, k \not = r\) is true for the expansion into powers of $k$, respectively. Then \(\lim_{n \rightarrow \infty}\sqrt[n]{W(r, k)} = r\) on $[r, r^{n + 1})$ and \(\lim_{m \rightarrow \infty}\sqrt[m]{W(r, k)} = k\) on $[k, k^{m + 1})$.
\end{theorem}
\begin{proof}
For the first limit and by substitution with the right hand side of Eqtn. (59) so that
\begin{equation}
\sqrt[n]{W(r, k)} = \sqrt[n]{b_{n}r^{n} + b_{n - 1}r^{n - 1} + \cdots + b_{0}},
\end{equation}
we get
\begin{eqnarray}
r^{n}&\leq           &b_{n}r^{n} + b_{n - 1}r^{n - 1} + \cdots + b_{0} < (n + 1)b_{n}r^{n}\\
     &\Longrightarrow&\lim_{n \rightarrow \infty}r \leq \lim_{n \rightarrow \infty} \sqrt[n]{b_{n}r^{n} + b_{n - 1}r^{n - 1} + \cdots + b_{0}}\\
     &\leq           &\lim_{n \rightarrow \infty}(n + 1)^{\frac{1}{n}}b_{n}^{\frac{1}{n}}r.
\end{eqnarray}
We have that
\begin{equation}
\lim_{n \rightarrow \infty}(n + 1)^{\frac{1}{n}}b_{n}^{\frac{1}{n}}r \leq r,
\end{equation}
since \(\lim_{n \rightarrow \infty}(n + 1)^{\frac{1}{n}} = 1\) and \(\lim_{n \rightarrow \infty}b_{n}^{\frac{1}{n}} \leq \lim_{n \rightarrow \infty}(r - 1)^{\frac{1}{n}} = 1\). Using these results in Eqtns. (63)--(64) gets us
\begin{eqnarray}
\lim_{n \rightarrow \infty}r&\leq           &\lim_{n \rightarrow \infty} \sqrt[n]{b_{n}r^{n} + b_{n - 1}r^{n - 1} + \cdots + b_{0}} \leq r\\
                            &\Longrightarrow&\lim_{n \rightarrow \infty} \sqrt[n]{b_{n}r^{n} + b_{n - 1}r^{n - 1} + \cdots + b_{0}} = \lim_{n \rightarrow \infty}\sqrt[n]{W(r, k}\\
                            &=              &r,
\end{eqnarray}
by substitution, by the fact that we have the limit
\begin{eqnarray}
\lim_{n \rightarrow \infty}\sqrt[n]{b_{n}r^{n} + b_{n - 1}r^{n - 1} + \cdots + b_{0}}&=   &r\lim_{n \rightarrow \infty}b_{n}^{\frac{1}{n}}\sqrt[n]{1 + \frac{b_{n - 1}}{b_{n}r} + \frac{b_{n - 2}}{b_{n}r^{2}} + \cdots + \frac{b_{0}}{b_{n}r^{n}}}\nonumber\\
                                                                                     &\leq&r\lim_{n \rightarrow \infty}(r - 1)^{\frac{1}{n}}\sqrt[n]{1 + \frac{b_{n - 1}}{b_{n}r} + \frac{b_{n - 2}}{b_{n}r^{2}} + \cdots + \frac{b_{0}}{b_{n}r^{n}}}\nonumber\\
                                                                                     &\leq&r,\nonumber
\end{eqnarray}
for the expression between the two inequality symbols in Eqtn. (66) and by application of the Pinching (``Squeeze") Theorem in Eqtns. (66)--(68).\\
\indent Similarly we obtain for the second limit in the Theorem and by substitution with the right hand side of Eqtn. (60) so that
\begin{equation}
\sqrt[m]{W(r, k)} = \sqrt[m]{c_{m}k^{m} + c_{m - 1}k^{m - 1} + \cdots + c_{0}},
\end{equation}
we get
\begin{eqnarray}
k^{m}&\leq           &c_{m}k^{m} + c_{m - 1}k^{m - 1} + \cdots + c_{0} < (m + 1)c_{m}k^{m}\\
     &\Longrightarrow&\lim_{m \rightarrow \infty}k \leq \lim_{m \rightarrow \infty} \sqrt[m]{c_{m}k^{m} + c_{m - 1}k^{m - 1} + \cdots + c_{0}}\\
     &\leq           &\lim_{m \rightarrow \infty}(m + 1)^{\frac{1}{m}}c_{m}^{\frac{1}{m}}k.
\end{eqnarray}
It follows then that
\begin{equation}
\lim_{m \rightarrow \infty}(m + 1)^{\frac{1}{m}}c_{m}^{\frac{1}{m}}k \leq k,
\end{equation}
since \(\lim_{m \rightarrow \infty}(m + 1)^{\frac{1}{m}} = 1\) and \(\lim_{m \rightarrow \infty}c_{m}^{\frac{1}{m}} \leq \lim_{m \rightarrow \infty}(k - 1)^{\frac{1}{m}} = 1\). Using these results in Eqtns. (71)--(72) gets us
\begin{eqnarray}
\lim_{m \rightarrow \infty}k&\leq           &\lim_{m \rightarrow \infty} \sqrt[m]{c_{m}k^{m} + c_{m - 1}k^{m - 1} + \cdots + c_{0}} \leq k\\
                            &\Longrightarrow&\lim_{m \rightarrow \infty} \sqrt[m]{c_{m}k^{m} + c_{m - 1}k^{m - 1} + \cdots + c_{0}}\\
                            &=              &\lim_{m \rightarrow \infty}\sqrt[m]{W(r, k)}\\
                            &=              &k,
\end{eqnarray}
again by substitution, by the fact that we have 
\begin{eqnarray}
\lim_{m \rightarrow \infty}\sqrt[m]{c_{m}k^{m} + c_{m - 1}k^{m - 1} + \cdots + c_{0}}&=   &k\lim_{m \rightarrow \infty}c_{m}^{\frac{1}{m}}\sqrt[m]{1 + \frac{c_{m - 1}}{c_{m}k} + \frac{c_{m - 2}}{c_{m}k^{2}} + \cdots + \frac{c_{0}}{c_{m}k^{m}}}\nonumber\\
                                                                                     &\leq&k\lim_{m \rightarrow \infty}(k - 1)^{\frac{1}{m}}\sqrt[m]{1 + \frac{c_{m - 1}}{c_{m}k} + \frac{c_{m - 2}}{c_{m}k^{2}} + \cdots + \frac{c_{0}}{c_{m}k^{m}}}\nonumber\\
                                                                                     &\leq&k,\nonumber
\end{eqnarray}
for the limit between the two inequality symbols in Eqtn. (74) and by application of the Pinching Theorem in Eqtns. (74)--(77).
\end{proof}
For Theorem 5.1 we could have considered
$$
r^{n} \leq W(r, k) < r^{n + 1} \Longrightarrow \lim_{n \rightarrow \infty}r \leq \lim_{n \rightarrow \infty} \sqrt[n]{W(r, k)} \leq \lim_{n \rightarrow \infty} r\cdot r^{\frac{1}{n}},
$$
and
$$
k^{m} \leq W(r, k) < k^{m + 1} \Longrightarrow \lim_{m \rightarrow \infty}k \leq \lim_{m \rightarrow \infty} \sqrt[m]{W(r, k)} \leq \lim_{m \rightarrow \infty} k\cdot k^{\frac{1}{m}},
$$
taking $n^{th}$ roots and $m^{th}$ roots respectively, then applying the Pinching theorem to derive the same two limits \(\lim_{n \rightarrow \infty} \sqrt[n]{W(r, k)} = r\), \(\lim_{m \rightarrow \infty} \sqrt[m]{W(r, k)} = k\), instead of proceeding with
$$
r^{n} \leq b_{n}r^{n} + b_{n - 1}r^{n - 1} + \cdots + b_{0} < (n + 1)b_{n}r^{n},
$$
then taking $n^{th}$ roots throughout, and
$$
k^{m} \leq c_{m}k^{m} + c_{m - 1}k^{m - 1} + \cdots + c_{0} < (m + 1)c_{m}k^{m}.
$$
then taking $m^{th}$ roots. The limits can be compared to two other previous results we obtained for large values of $\log_{r}W(r, k)$ and $\log_{k}W(r, k)$ in Theorem 2.1, namely \(\log_{r}W(r, k) = n + O(1)\) and \(\log_{k}W(r, k) = m + O(1)\) (See also~\cite{Betts1}). 
\section{One more Theorem, a Corollary and Concluding Remarks}
In this concluding Section we provide one more Theorem and a Corollary, to expand further the results on $W(2, 7)$ in Section 3.
\begin{theorem}
Let
\begin{eqnarray}
W(2, 7)&=&2^{n} + b_{n - 1}\cdot 2^{n - 1} + \cdots + b_{0} \in [2^{n}, 2^{n + 1}),\\
W(2, 7)&=&c_{m}7^{m} + c_{m - 1}\cdot 7^{m - 1} + \cdots + c_{0} \in [7^{m}, 7^{m + 1}),
\end{eqnarray}
For some nonnegative integers
\begin{eqnarray}
b_{n} = 1, \: b_{n - 1}, \ldots, b_{0} \in \{0, 1\},\\
c_{m} \in [1, 6], \: \: c_{m - 1}, \ldots, c_{0} \in [0, 6] \subset \mathbb{R}.
\end{eqnarray}
Then \(n \geq 11\) and \(m \geq 4\).
\end{theorem}
\begin{proof}
Recall that, from J. Rabung and M. Lotts~\cite{Rabung and Lotts}, \(W(2, 7) > 3703\). So for the positive integer exponent $n$,
\begin{eqnarray}
2^{11}&<              &3703 < W(2, 7) < 2^{n + 1}\\
      &\Longrightarrow&11 < \log_{2}3703 < \log_{2}W(2, 7) < (n + 1)\log_{2}2,\nonumber\\
      &\Longrightarrow&11 < \log_{2}3703 < n + 1,\\
      &\Longrightarrow&10 < \log_{2}3703 - 1 < n\nonumber \\
      &\Longrightarrow&10 < 10.8544\cdots < n\\
      &\Longrightarrow&11 \leq n.
\end{eqnarray}
For the positive integer exponent $m$,
\begin{eqnarray}
7^{4}&<              &3703 < W(2, 7) < 7^{m + 1}\\
     &\Longrightarrow&4 < \log_{7}3703 < \log_{7}W(2, 7) < (m + 1)\log_{7}7,\nonumber\\
     &\Longrightarrow&4 < \log_{7}3703 < m + 1,\\
     &\Longrightarrow&3 < \log_{7}3703 - 1 < m\nonumber \\
     &\Longrightarrow&3 < 3.22267\cdots < m\\
     &\Longrightarrow&4 \leq m.
\end{eqnarray}
Hence \(n \geq 11\) and \(m \geq 4\).
\end{proof}
We just have shown that, for van der Waerden number $W(2, 7)$, \(n \geq 11\). Letting \(k = 7, r = 2\) and as in Theorem 6.1,
\begin{equation}
W(2, 7) = 2^{n} + b_{n - 1}\cdot 2^{n - 1} + \cdots + b_{0}, 
\end{equation}
\(b_{n - 1}, \ldots, b_{0} \in \{0, 1\}\), recall Statement (3) in Section 1, as well as our previous results~\cite{Betts1},~\cite{Betts2}. So if the condition in Statement 3, Section 1 in the Introduction holds for $W(2, 7)$--and in fact this condition in Statement 3, Section 1 \emph{does hold already} for all the van der Waerden numbers 
$$
W(2,3), W(2, 4), W(2, 5), W(2, 6), W(3, 3), W(3, 4), W(4, 3),
$$
known so far~\cite{Betts1}~\cite{Betts2} (See Table 2, this paper), and since we have also the fact that \(3703 < W(2, 7) < 2^{n + 1}\) (See Eqtns. (2)--(3), in Section 1, and the example shown with $W(2, 6)$ in Eqtn. (5), Section 1),
\begin{equation}
3703 < W(2, 7) < 2^{n + 1} \leq 2^{49},  
\end{equation}
is true provided \(7 \geq \sqrt{n + 1} \Leftrightarrow n \leq 48\) \(\Longrightarrow n \in [11, 48]\).\\
\begin{corollary}
Suppose Statement (3), Section 1, is true for the van der Waerden number $W(2, 7)$. That is, for some positive integer exponent \(n \geq 11\), such that
\begin{equation}
W(2, 7) \in [2^{n}, 2^{n + 1}),
\end{equation}
where $W(2, 7)$ has the expansion as in Eqtn. (90) with \(b_{n} = 1\), \(b_{n - 1}, \ldots, b_{0} \in \{0, 1\}\), suppose
\begin{equation}
W(2, 7) < 2^{n + 1} \leq 2^{49} \Leftrightarrow (7 \geq \sqrt{n + 1} \Leftrightarrow n \leq 48),
\end{equation}
holds also for $W(2, 7)$, just as Statement (3), Section 1 holds already for $W(2, 3)$, $W(2, 4)$, $W(2, 5)$, $W(2, 6)$, $W(3, 3)$, $W(3, 4)$ and $W(4, 3)$. Then \(W(2, 7) \in [2^{11}, 2^{48})\).
\end{corollary}
\begin{proof}
This follows at once by applying first Theorem 6.1 which establishes that, at the very least, \(n \geq 11\), then Statement (3) (in the Introduction), to $W(2, 7)$.
\end{proof}
So if Statement (3), Section 1 in the Introduction does apply to $W(2, 7)$, computational methods, perhaps with the use of some SAT algorithm, one day ought to verify that \(2^{11} \leq W(2, 7) < 2^{49}\). That is, if the Corollary holds,
\begin{equation}
W(2, 7) \in [3703, 281474976710656] \subset [2048, 5629499534421312) \subset \mathbb{R}.
\end{equation}
\indent Note that \(2^{12} = 4096\), and that \(\sqrt[12]{3703} = 1.983259\cdots, \sqrt[12]{4096} = 2\), so in comparison to these results and using Theorem 5.1, the integer exponent $n$ has a value such that $\sqrt[n]{W(2, 7)}$ ought to be close to the value two. \\ 
\indent Given any $r$, $k$, is it possible to develop a fast algorithm that can approximate or estimate the value of the integer exponent $n$, such that
\begin{equation}
W(r, k) \in [r^{n}, r^{n + 1})?
\end{equation}
If there does exist such an algorithm it can reduce the computational time and bit compexity required to determine not only the value for any unknown $W(r, k)$ such as $W(2, 7)$, but also for the lower bound $r^{n}$.  
\pagebreak
\appendix
\section{A necessary and sufficient Condition for which \(W(r, k) < r^{k^{2}}\)}
The Proof to Theorem A.1 is very straightforward. Therefore we have chosen to put it in an Appendix instead of in the main part of the paper.\\
\indent Here we provide the proof for Statement 3, Section 1, in the Introduction. Note first however, that if 
$$
k^{2} \leq n \Longrightarrow r^{k^{2}} \leq W(r, k) < r^{n + 1},
$$
actually does occur for some van der Waerden numbers that are unknown at present, then this condition \(k^{2} \leq n\) \emph{does not occur always for each and every} $W(r, k)$ known and unknown, because it certainly does fail to hold for all the known van der Waerden numbers we know and that are listed in Table 2. We know this because we can compute the value of the exponent $n$ for each of these then compare the size of $k$ with the size of $\sqrt{n + 1}$ for each of these known van der Waerden numbers (See Table 2). Also the condition \(n < k^{2} < n + 1\) such that \(n < k^{2} < n + 1 \Longrightarrow r^{n} < r^{k^{2}} < r^{n + 1}\) is impossible because the interval $(n, n + 1)$ is open on $\mathbb{R}$ and contains no integers.\\
\indent The derivations in Eqtns. (98)--(100) and Eqtns. (100)--(102) are based on the known properties of the powers of an integer having positive integer exponents and the properties of logarithms. 
\begin{theorem}
Let \(W(r, k) \in [r^{n}, r^{n + 1}) \in \mathbb{R}\), for the reasons discussed already in Section 1, meaning
\begin{equation}
W(r, k) = b_{n}r^{n} + b_{n - 1}r^{n - 1} + \cdots + b_{0} \in [r^{n}, r^{n + 1}).
\end{equation}
Then a necessary and sufficient condition for \(W(r, k) < r^{k^{2}}\) to hold is
\begin{equation}
(k \geq \sqrt{n + 1} \Leftrightarrow n \leq k^{2} - 1).
\end{equation}
\end{theorem}
\begin{proof}
\emph{Sufficiency}: Suppose \(k \geq \sqrt{n + 1} \Leftrightarrow n \leq k^{2} - 1\) holds. Then
\begin{eqnarray}
(k \geq \sqrt{n + 1} \Leftrightarrow n \leq k^{2} - 1)&\Longrightarrow&(k \geq \sqrt{n + 1} \Longrightarrow n \leq k^{2} - 1)\\
                                                      &\Longrightarrow&r^{n} \leq r^{k^{2} - 1}\nonumber\\
                                                      &\Longrightarrow&r^{n + 1} \leq r^{k^{2}}\nonumber\\
                                                      &\Longrightarrow&r^{\log_{r}W(r, k)} < r^{n + 1} \leq r^{k^{2}}\\
                                                      &\Longrightarrow&W(r, k) < r^{n + 1} \leq r^{k^{2}},
\end{eqnarray}
where in Eqtns. (99)--(100), \(W(r, k) = r^{\log_{r}W(r, k)}\) and \(r^{n} \leq W(r, k) < r^{n + 1}\) was given. Therefore it suffices for \(k \geq \sqrt{n + 1} \Leftrightarrow n \leq k^{2} - 1\) to be true in order for \(W(r, k) < r^{n + 1} \leq r^{k^{2}}\) to be true.\\
\noindent \emph{Necessity}: 
\begin{eqnarray}
r^{n}&\leq W(r, k)   &< r^{n + 1} \leq r^{k^{2}}\\
     &\Longrightarrow&n \leq \log_{r}W(r, k) < n + 1 \leq k^{2}\nonumber\\
     &\Longrightarrow&n - 1 \leq \log_{r}W(r, k) - 1 < n \leq k^{2} - 1.
\end{eqnarray}
From this argument we derive \(W(r, k) < r^{n + 1} \leq r^{k^{2}} \Longrightarrow n \leq k^{2} - 1\). But since \(n \leq k^{2} - 1\) we can add one to both sides of this inequality then take square roots of both sides to derive \(\sqrt{n + 1} \leq k\), to show that \(n \leq k^{2} - 1 \Longrightarrow \sqrt{n + 1} \leq k\). Conversely we can square both sides of the inequality sign in \(\sqrt{n + 1} \leq k\) then subtract one from both sides to get \(n \leq k^{2} - 1\), to show that \(\sqrt{n + 1} \leq k \Longrightarrow n \leq k^{2} - 1\). Thus we have shown that
\begin{equation}
W(r, k) < r^{n + 1} \leq r^{k^{2}} \Longrightarrow (k \geq \sqrt{n + 1} \Leftrightarrow n \leq k^{2} - 1),
\end{equation}
which means \(W(r, k) < r^{n + 1} \leq r^{k^{2}}\) is true only if \((k \geq \sqrt{n + 1} \Leftrightarrow n \leq k^{2} - 1)\) is true.
\end{proof}
\pagebreak
\appendix
\section{The Possible Intervals $[2^{n}, 2^{n + 1})$ that contain $W(2, 7)$, upon Application of Theorem A.1}
Theorem A.1 does apply to the known van der Waerden numbers $9$, $35$, $178$, $1132$, $27$, $293$ and $76$, as the reader can verify with Table 2. So using \(W(r, k) \in [r^{n}, r^{n + 1})\) and provided Theorem A.1 applies such that for some exponent \(n \in [11, 48]\), one of the following subsets contains $W(2, 7)$,
\begin{eqnarray}
&[2^{11}, 2^{12}), [2^{12}, 2^{13}), [2^{13}, 2^{14}),&[2^{14}, 2^{15}), [2^{15}, 2^{16}), [2^{16}, 2^{17}),\\ 
&[2^{17}, 2^{18}), [2^{18}, 2^{19}), [2^{19}, 2^{20}),&[2^{20}, 2^{21}), [2^{21}, 2^{22}), [2^{22}, 2^{23}),\nonumber\\
&[2^{23}, 2^{24}), [2^{24}, 2^{25}), [2^{25}, 2^{26}),&[2^{26}, 2^{27}), [2^{27}, 2^{28}), [2^{28}, 2^{29}),\\ 
&[2^{29}, 2^{30}), [2^{30}, 2^{31}), [2^{31}, 2^{32}),&[2^{32}, 2^{33}), [2^{33}, 2^{34}), [2^{34}, 2^{35}),\nonumber\\
&[2^{35}, 2^{36}), [2^{36}, 2^{37}), [2^{37}, 2^{38}),&[2^{38}, 2^{39}), [2^{39}, 2^{40}), [2^{40}, 2^{41}),\\ 
&[2^{41}, 2^{42}), [2^{42}, 2^{43}), [2^{43}, 2^{44}),&[2^{44}, 2^{45}), [2^{45}, 2^{46}), [2^{46}, 2^{47}),\nonumber\\
&[2^{47}, 2^{48}), [2^{48}, 2^{49}).                  &
\end{eqnarray}

\pagebreak

\end{document}